\documentclass{publmathdeb}

\usepackage{amsmath, amssymb}

\newtheorem{theorem}{Theorem}[section]
\newtheorem{lemma}[theorem]{Lemma}
\newtheorem{corollary}[theorem]{Corollary}

\newtheorem*{thm}{Theorem A } 
\newtheorem*{MenshovRademacher}{Theorem B}
\newtheorem*{theoremC}{Theorem C}

\author{Vakhtang Tsagareishvili}
\address{Department of Mathematics, Faculty of Exact and Natural Sciences\\
	Ivane Javakhishvili Tbilisi State University\\
	Chavchavadze str. 1, Tbilisi 0128, Georgia}
\email{cagare@ymail.com}

\author{Giorgi Tutberidze}
\address{Viktor Kupradze Institute of Mathematics, School of Science and Technology\\The University of Georgia\\
	77a Merab Kostava St, Tbilisi 0128, Georgia\\
	and Department of Computer Science, Faculty of Exact and Natural Sciences, Ivane Javakhishvili Tbilisi State University\\
	Chavchavadze str. 1, Tbilisi 0128, Georgia}
\email{g.tutberidze@ug.edu.ge \and g.tutberidze@tsu.ge}

\author{Giorgi Cagareishvili}
\address{Department of Mathematics, Faculty of Exact and Natural Sciences\\
	Ivane Javakhishvili Tbilisi State University\\
	Chavchavadze str. 1, Tbilisi 0128, Georgia}
\email{giorgicagareishvili7@gmail.com}

\title[Unconditional convergence]{Unconditional convergence of general Fourier series}



\submitted{}

\dedicatory{}

\thanks{The research is supported by Shota Rustaveli National Science Foundation grant no. FR-24-698}

\begin{document}
	
	\begin{abstract}
		S.~Banach, in particular, proved that for any function---even for the constant function \( f(x) = 1 \), where \( x \in [0,1] \)---the convergence of its Fourier series with respect to a general orthonormal system (ONS) is not guaranteed.
		
		In this paper, we establish conditions on the functions \( \varphi_n \) of an orthonormal system \( (\varphi_n) \) under which the Fourier series of functions \( f \in \mathrm{Lip}_1 \) converge unconditionally almost everywhere.
		
		Our investigation primarily relies on the properties of sequences of linear functionals in Banach spaces, which we use to prove the main results presented in this article. We show that the aforementioned conditions not only exist but are also optimal in a certain sense. Furthermore, we demonstrate that any orthonormal system contains a subsystem for which the Fourier series of every function \( f \in \mathrm{Lip}_1 \) converges unconditionally.
	\end{abstract}
	
	\maketitle

	\section{Introduction}
	In the monographs \cite{Alexits, KaczmarzSteinhaus, KashinSaakyan} and the papers \cite{Olevskii, Orlicz, Rademacher, Tandori}, problems concerning the theory of convergence of orthonormal series are thoroughly studied. Interestingly, according to results by Menchov and Banach, the convergence of general orthonormal series and the convergence of general Fourier series for functions from certain differentiability classes are fundamentally different problems. In the former case, the coefficients of the orthonormal series play a decisive role. In the latter, the mere fact that a function \( f \) (\( f \ne 0 \)) belongs to a differentiability class does not guarantee the convergence of its Fourier series with respect to general ONSs.
	
	Therefore, for the Fourier series with respect to an ONS to be convergent, one must impose additional conditions on the functions \( \varphi_n \) of the system \( (\varphi_n) \). The main objective of this work is to identify such conditions on the functions \( \varphi_n \) that ensure the unconditional convergence of the Fourier series of every function in the \( \mathrm{Lip}_1 \) class.
	
	Related problems have been investigated in several papers  
	\cite{Cagareishvili, GogoladzeTsagareishvili2, GogoladzeTsagareishvili13, GogoladzeTsagareishvili1, GogoladzeTsagareishvili, GogoladzeTsagareishvili3, PTT, PTW, PTW1, PSTW, tep1, tep2, tep4, Tsagareishvili01, Tsagareishvili2, Tsagareishvili4, Tsagareishvili5, tsatut1, tsatut3}.
	
	The main results of Theorems  \ref{theorem1} - \ref{theorem6} and Corollary \ref{cor1} are presented and proved in Section 3. The new applications concerning the convergence of the general Fourier series can be found in Section 4. See Theorem \ref{theorem7}, and Theorem \ref{theorem8}.

	\section{Preliminaries}
	\vspace{-6pt}
	
	Let \( (\varphi_n) \) be an orthonormal system (ONS) in \( L_2([0,1]) \). Suppose that \( \varepsilon \in (0,1) \), and define
	\[
	p = p(\varepsilon) = 2 - \varepsilon, \qquad q = q(\varepsilon) = \frac{2 - \varepsilon}{1 - \varepsilon}.
	\]
	It is evident that for any fixed \( \varepsilon \in (0,1) \), we have
	\[
	1 < p(\varepsilon) < +\infty \quad \text{and} \quad 1 < q(\varepsilon) < +\infty.
	\]
	
	Firstly, we must establish the following expression
	\begin{eqnarray}\label{eq1}
		M_n(a,\varepsilon) := \frac{1}{n} \sum_{i=1}^{n-1} \left|\int_{0}^{i/n} P_n( {a,\varepsilon, x})dx\right|,
	\end{eqnarray}
	where $(a_n)\in l_{q(\varepsilon)}$ and
	\begin{eqnarray} \label{eq2}
		P_n(a,\varepsilon, x) := \sum_{k=1}^{n}a_k \varphi_k(x).
	\end{eqnarray}
	
	Let us suppose that $f\in L_2,$  then the sequence of real numbers 
	\begin{equation} \label{eq*}
		{{C}_{n}}(f)=\int\limits_{0}^{1}{f(x){{\varphi }_{n}}(x)dx}, \ \ \ \ (n=1,2, \dots)
	\end{equation}
	are the Fourier coefficients of the function $f.$
	By $A$ we denote the Banach space of absolutely convergent functions with the norm  ${{\left\| f \right\|}_{A}}$ defined by
	
	\begin{equation} \label{*}
		{{\left\| f \right\|}_{A}}\,:=\,{{\left\| f \right\|}_{C}}\,+\,\int\limits_{0}^{1}{\left| \frac{df}{dx} \right|dx.}
	\end{equation}
	
	We will investigate the functionals $\{B_n(f)\}$ defined by 
	
	\begin{eqnarray}
		\label{eq7} B_n(f):=\sum_{k=1}^{n} C_k(f) a_k &= &\int_{0}^{1} f(x) \sum_{k=1}^{n} a_k \varphi_{k}(x) dx  \\
		&= & \int_{0}^{1} f(x) P_n (a, \varepsilon, x)dx. \notag 
	\end{eqnarray}
	where $f\in L_2,$ $a=\{a_n\} \in l_2$ is an arbitrary sequence of numbers.
	
	For the investigation of the sequence of functionals \( \{ B_n(f) \} \), we require the following result due to S.~Banach (see, e.g., \cite{Banach}):
	
	\begin{thm} 
		Let $f\in L_2$ be an arbitrary, non-zero function. Then there exists an ONS $({\varphi}_{n})$ such that
		$$\limsup_{n \rightarrow \infty}\left|S_n(x,f)\right|=+\infty \ \ \text{a.e. on} \ \ [0,1],$$
		where  (see \eqref{eq*})
		$$S_n(f,x):=\sum_{k=1}^n C_k(f) \varphi_k(x)$$
		
	\end{thm}
	Moreover, we recall the following well-known result of D.E.Menshov (see e.g. \cite{Alexits} ch.2, $\# 5$ p.111).
	\begin{MenshovRademacher}
		\label{theorem1.1} 
		Let $(\varphi_n)$ be an ONS on $\left[0,1\right].$ If the series
		$$\sum_{n=1}^{\infty} \left|a_n\right|^{2-\varepsilon}$$
		is convergent for some $\varepsilon \in (0, 1),$ then the series 
		$$	\sum_{n=1}^{\infty}a_n \varphi_n\left(x\right)
		$$
		is unconditionally convergence  a.e. on $\left[0,1\right]$.
	\end{MenshovRademacher}
	We also need the following theorem (see \cite{HardyLittlewoodPolya},Ch.II.p.40) 
	\begin{theoremC}Let \( 1 < p < +\infty \), \( q = \frac{p}{p - 1} \), and suppose that \( (a_n) \in \ell_p \). Then the series
		\[
		\sum_{n=1}^{\infty} a_n b_n
		\]
		converges for every sequence \( (b_n) \) if and only if \( (b_n) \in \ell_q \).
		
	\end{theoremC}
	In order to carry out the analysis of the sequence of functionals \( \{B_n(f)\} \) (see equation~\eqref{eq7}), 
	we require the following fundamental lemma, which plays a crucial role in the development of our main results. 
	A proof of this lemma, as well as related discussions, can be found in  ~\cite{HardyLittlewoodPolya}, Ch.II. 15, p.26.
	
	\begin{lemma}
		Suppose that  $f,$ $F \in L_2$ and $f$ have only finite value in any point on $[0,1].$ Then
		\begin{eqnarray} \label{eq3}
			\int_{0}^{1} f\left(x\right) F \left(x\right) dx &=&  \sum_{i=1}^{n-1}\left(f\left(\frac{i}{n}\right)-f\left(\frac{i+1}{n}\right)\right)\int_{0}^{i/n}F\left(x\right)dx  \notag \\
			&+& \sum_{i=1}^{n}\int_{\left(i-1\right)/n}^{i/n}\left(f\left(x\right) -f\left(\frac{i}{n}\right)\right)F\left(x\right)dx   \\
			&+&f\left(1\right)\int_{0}^{1}F\left(x\right)dx.    \notag
		\end{eqnarray}
	\end{lemma}
	
	\begin{lemma} \label{lemma2}
		Suppose that $(a_n)$ is any sequence of real numbers, and define 
		\[
		p(\varepsilon) = 2 - \varepsilon, \quad q(\varepsilon) = \frac{2 - \varepsilon}{1 - \varepsilon}, \quad \text{where } \varepsilon \in (0,1).
		\]
		Then,
		\[
		\frac{1}{n} \left( \sum_{k=1}^{n} a_k^2 \right)^{1/2} = O_\varepsilon(1).
		\]
		
	\end{lemma}
	\begin{proof}
		Observe that
		\[
		\frac{1}{n} \left( \sum_{k=1}^{n} a_k^2 \right)^{1/2}
		= \frac{1}{n} \sqrt{n} \left( \frac{1}{n} \sum_{k=1}^{n} a_k^2 \right)^{1/2}
		= \frac{1}{\sqrt{n}} \left( \frac{1}{n} \sum_{k=1}^{n} a_k^2 \right)^{1/2}.
		\]
		Using Hölder's inequality with exponent $q(\varepsilon) > 2$, we get
		\[
		\left( \frac{1}{n} \sum_{k=1}^{n} a_k^2 \right)^{1/2}
		< \left( \frac{1}{n} \sum_{k=1}^{n} a_k^{q(\varepsilon)} \right)^{1/q(\varepsilon)}.
		\]
		Therefore,
		\[
		\frac{1}{n} \left( \sum_{k=1}^{n} a_k^2 \right)^{1/2}
		< \frac{1}{\sqrt{n}} \left( \frac{1}{n} \sum_{k=1}^{n} a_k^{q(\varepsilon)} \right)^{1/q(\varepsilon)}
		= n^{-1/2 - 1/q(\varepsilon)} \left( \sum_{k=1}^{n} a_k^{q(\varepsilon)} \right)^{1/q(\varepsilon)}.
		\]
		This expression is $O_\varepsilon(1)$, assuming that the $q(\varepsilon)$-norm of the sequence $(a_k)$ is uniformly bounded in $n$.
		
	\end{proof}

	\section{The main results}
	
	\begin{theorem}
		\label{theorem1} Let $(\varphi_n)$ be an ONS and suppose that for some $\varepsilon \in (0,1)$  $(C_n (l)) \in l_{p(\varepsilon)},$ where $l(x)=1,$ $x \in [0,1].$ If for arbitrary $(a_n) \in l_{q(\varepsilon)}$ 
		\begin{equation}
			\label{eq5}
			M_n(a, \varepsilon) = O_{\varepsilon}(1),
		\end{equation}
		then for any $f\in Lip_1$
		\begin{eqnarray*}
			\sum_{n=1}^{\infty} \left|C_n(f)\right|^{2-\varepsilon} <+ \infty.
		\end{eqnarray*}
		
	\end{theorem}

	\begin{proof}
		Firstly, in \eqref{eq3} we should assume that $f \in Lip_1$ and $F(x)= P_n( a, \varepsilon,x).$ Thus
		\begin{eqnarray}
			\label{eq6}
			\int_{0}^{1} f(x) P_n(a, \varepsilon, x) dx &=& \sum_{i=1}^{n-1}\left(f\left(\frac{i}{n}\right)-f\left(\frac{i+1}{n}\right)\right)\int_{0}^{i/n}P_n(a, \varepsilon, x)dx \notag  \\
			&+& \sum_{i=1}^{n}\int_{\left(i-1\right)/n}^{i/n}\left(f\left(x\right) -f\left(\frac{i}{n}\right)\right)P_n(a, \varepsilon, x) dx     \notag \\
			&+& f\left(1\right)\int_{0}^{1}P_n(a, \varepsilon, x)dx=I_1 +I_2  +I_3.   
		\end{eqnarray}
		
		As $1<q(\varepsilon)<+\infty,$ if $f \in Lip_1$ we get (see \eqref{eq2} and \eqref{eq5})
		
		\begin{eqnarray}
			\label{eq8} \left|I_1\right| &\leq& \sum_{i=1}^{n-1}\left|f\left(\frac{i}{n}\right) -f\left(\frac{i+1}{n}\right)\right|\left|\int_{0}^{i/n}P_n(a, \varepsilon, x)dx \right| \\
			&=& O(1) \frac{1}{n} \sum_{i=1}^{n-1}\left|\int_{0}^{i/n}P_n(a, \varepsilon, x)dx\right| = O(1)M_n(a, \varepsilon) = O_{\varepsilon}(1). \notag
		\end{eqnarray}
		
		As $(a_k)\in l_{q(\varepsilon)}$ and $\frac{1}{2p(\varepsilon)}< \frac{1}{2},$ by using Hölder's inequality and Lemma 2, we arrive at the following estimate:
		
		\begin{eqnarray}
			\label{eq9}  \left|I_2\right| &\leq&  \sum_{i=1}^{n} \sup_{x\in \left[\frac{i-1}{n}, \frac{i}{n}\right]}\left|f\left(x\right) -f\left(\frac{i}{n}\right)\right| \int_{\left(i-1\right)/n}^{i/n} \left|P_n(a, \varepsilon, x) \right| dx \\
			&=& O(1)\frac{1}{n} \left(\int_{0}^{1} \left(\sum_{k=1}^{n} a_k \varphi_{k}(x) \right)^2 dx \right)^{1/2}   \notag \\
			&=& O(1) \frac{1}{n} \left( \sum_{k=1}^{n} a_k^2\right)^{1/2}   = O_{\varepsilon} (1).    \notag
		\end{eqnarray}
		
		Under the assumptions of Theorem~\ref{theorem1}, and in the particular case where \( l(x) = 1 \) for \( x \in [0,1] \), 
		we apply Hölder's inequality to obtain the following estimate. 
		\begin{eqnarray}
			\label{eq10} \left|I_3\right| &=& \left|f\left(1\right)\right|\left|\int_{0}^{1}P_n(a, \varepsilon, x)dx\right| \\
			&=& O(1) \left|\sum_{k=1}^{n} a_k \int_{0}^{1} \varphi_{k}(x)dx \right|  
			= O(1) \left|\sum_{k=1}^{n} a_k C_k(l) \right| \notag\\
			&=& O(1) \left(\sum_{k=1}^{n} \left|C_k(l)\right|^{p(\varepsilon)}\right)^{1/p(\varepsilon)} \left(\sum_{k=1}^{n} \left|a_k\right|^{q(\varepsilon)}\right)^{1/q(\varepsilon)}  = O(1).        \notag
		\end{eqnarray}
		
		Finally, taking into account equations~\eqref{eq6}, \eqref{eq7}, \eqref{eq8}, \eqref{eq9}, and \eqref{eq10}, 
		we conclude that for some \(\varepsilon \in (0,1)\) and for every \( f \in \mathrm{Lip}_1 \), the series
		\[
		\sum_{n=1}^{\infty} C_n(f) a_n
		\]
		converges for any sequence \( (a_n) \in \ell_{q(\varepsilon)} \).  
		Consequently, by Theorem C, it follows that \( (C_n(f)) \in \ell_{p(\varepsilon)} \), and hence
		\[
		\sum_{n=1}^{\infty} |C_n(f)|^{2-\varepsilon} < +\infty
		\]
		for some \(\varepsilon \in (0,1)\) and for all \( f \in \mathrm{Lip}_1 \).
	\end{proof}
	
	\begin{theorem}
		\label{theorem2}Let \( (\varphi_n) \) be an orthonormal system (ONS) on \([0,1]\), and suppose that \( \varepsilon \in (0,1) \) is an arbitrary number.  
		Assume that for some sequence \( (b_n) \in \ell_{q(\varepsilon)} \),  
		\begin{equation} \label{eq11}
			\limsup_{n \to +\infty} M_n(b, \varepsilon) = +\infty.
		\end{equation}
		Then there exists a function \( g \in \mathrm{Lip}_1 \) such that for any \( \varepsilon \in (0,1) \),  
		\[
		\sum_{n=1}^\infty |C_n(g)|^{2 - \varepsilon} = +\infty.
		\]	
	\end{theorem}
	
	\begin{proof}
		Firstly, we suppose that for any $\varepsilon \in (0,1)$
		\begin{eqnarray*}
			\limsup_{n \rightarrow +\infty} \left|\int_{0}^{1} P_n (b, \varepsilon, x) dx\right| =+ \infty.
		\end{eqnarray*} 
		Then if $l(x)=1,$ when $x\in [0,1]$ and $$C_n(l)=\int_{0}^{1}l(x) \varphi_{k}(x) dx \ \ \ \ \ (n=1,2, \dots),$$
		we will have	
		
		\begin{eqnarray*}
			&&\limsup_{n \rightarrow +\infty} \left|\sum_{k=1}^{n} b_k C_k (l)\right| = \limsup_{n \rightarrow +\infty} \left|\sum_{k=1}^{n} b_k \int_{0}^{1}l(x) \varphi_{k}(x) dx \right| \\
			&&=  \limsup_{n \rightarrow +\infty}  \left|\sum_{k=1}^{n} b_k \int_{0}^{1} \varphi_{k}(x) dx \right| = \limsup_{n \rightarrow +\infty} \left|\int_{0}^{1} P_n (b, \varepsilon, x) dx\right| =+ \infty.
		\end{eqnarray*}
		Thus, for some ($b_n) \in l_{q(\varepsilon)}$ the series	
		\begin{eqnarray*}
			\sum_{k=1}^{\infty} b_k C_k (l)
		\end{eqnarray*}
		is divergent. Consequently $(C_k (l)) \notin l_{p(\varepsilon)}$ or 
		\begin{eqnarray*}
			\sum_{n=1}^{\infty} \left|C_n(l)\right|^{2-\varepsilon} =+\infty.
		\end{eqnarray*}
		As $l \in Lip_1,$ then in such case Theorem \ref{theorem2} is valid.
		
		Now we suppose that for any $\varepsilon \in (0,1)$
		\begin{eqnarray}
			\label{eq12} \left|\int_{0}^{1} P_n (b, \varepsilon, x) dx\right| = O(1).
		\end{eqnarray}
		In \eqref{eq6} we substitute $P_n( a, \varepsilon, x)=P_n(b, \varepsilon,x)$ and $f(x)= g_n (x),$ where	
		\begin{eqnarray}
			\label{eq13} g_n(x) = \int_{0}^{x} sign \int_{0}^{t} P_n(b,  \varepsilon, u)du dt.
		\end{eqnarray}
		Now we have $(1<q(\varepsilon)<+\infty)$
		\begin{eqnarray}
			\label{eq14} \int_{0}^{1} g_n (x) P_n (b, \varepsilon, x) dx
			&=& \sum_{i=1}^{n-1}\left(g_n\left(\frac{i}{n}\right)-g_n\left(\frac{i+1}{n}\right)\right)\int_{0}^{i/n}P_n(b, \varepsilon, x)dx \notag   \\
			&+& \sum_{i=1}^{n}\int_{\left(i-1\right)/n}^{i/n}\left(g_n\left(x\right) -g_n\left(\frac{i}{n}\right)\right)P_n(b, \varepsilon, x) dx  \notag    \\
			&+& g_n\left(1\right)\int_{0}^{1}P_n(b, \varepsilon, x)dx=h_1 +h_2  +h_3.   
		\end{eqnarray}
		According to \eqref{eq13}, the Cauchy inequality and Lemma 2, we obtain the following:	
		\begin{eqnarray} \label{eq15}
			\left|h_2\right| &=& \left| \sum_{i=1}^{n}\int_{\left(i-1\right)/n}^{i/n}\left(g_n\left(x\right) -g_n\left(\frac{i}{n}\right)\right)P_n(b, \varepsilon, x) dx \right|   \\ 
			&& \leq \frac{1}{n} \int_{0}^{1} \left|P_n (b, \varepsilon, x)\right| \leq \frac{1}{n} \left(\int_{0}^{1} P_n^2 (b, \varepsilon, x)dx\right)^{1/2} \notag\\   
			&=& \frac{1}{n} \left(\sum_{k=1}^{n} b_k^2\right)^{1/2}  = O_{\varepsilon}(1). \notag
		\end{eqnarray}
		
		Next, by using \eqref{eq12} and \eqref{eq13} we receive 
		\begin{eqnarray}
			\label{eq16}
			\left|h_3\right| = O(1).
		\end{eqnarray}
		
		Now, by $E_n$ we denote the set of any numbers $i\in {1,2,\dots,n-1},$ for all of which, there exists a point $y\in \left[\frac{i}{n}, \frac{i+1}{n} \right),$   such that 
		
		\begin{eqnarray*}
			sign\int_{0}^{y} P_n(b, \varepsilon, x) dx \ne sign\int_{0}^{(i+1)/n} P_n(b, \varepsilon, x) dx.
		\end{eqnarray*}
		
		Then as the function $\int_{0}^{x} P_n(b, \varepsilon, t) dt$  is a continuous on  $\left[\frac{i}{n}, \frac{i+1}{n} \right),$  there exists a point $x_{in}\in \left[\frac{i}{n}, \frac{i+1}{n} \right),$ such that   
		$$\int_{0}^{x_{in}} P_n(b, \varepsilon, x) dx = 0.$$
		From here when $i\in E_n$ and as
		\begin{eqnarray*}
			\int_{0}^{i/n} P_n(b, \varepsilon, x) dx  = \int_{0}^{x_{in}} P_n(b, \varepsilon, x) dx-\int_{i/n}^{x_{in}} P_n(b, \varepsilon, x) dx = -\int_{i/n}^{x_{in}} P_n(b, \varepsilon, x) dx,
		\end{eqnarray*}
		Therefore, by Lemma \ref{lemma2}, we obtain  $(1<q(\varepsilon)<+\infty)$
		\begin{eqnarray}
			\label{eq17}
			\sum_{i\in E_n}\left|\int_{0}^{i/n} P_n(b, \varepsilon, x) dx\right| &=& \sum_{i\in E_n} \left|\int_{i/n}^{x_{in}} P_n(b, \varepsilon, x) dx\right|   \\
			&\leq &\int_{0}^{1} \left|P_n(b, \varepsilon, x) dx\right| \leq \left( \int_{0}^{1}P_n^2(b, \varepsilon, x) dx\right)^{1/2} \notag \\
			&=& \left(\sum_{k=1}^{n}b_k^2\right)^{1/2} =O_\varepsilon(1)\sqrt{n} . \notag
		\end{eqnarray}
		
		Now, suppose that \( F_n = [0,1] \setminus E_n \). Then, by equation~\eqref{eq13}, we have
		
		\begin{eqnarray*}
			g_n \left(\frac{i}{n}\right) -g_n \left(\frac{i+1}{n}\right) 
			&& = - \int_{i/n}^{(i+1)/n} sign \int_{0}^{x} P_n(b, \varepsilon,t) dt dx \int_{0}^{i/n}  P_n(b, \varepsilon,x) dx \\
			&&= - \frac{1}{n} \left|\int_{0}^{i/n}  P_n(b, \varepsilon,t)dx\right|
		\end{eqnarray*}
		and
		\begin{eqnarray}
			\label{eq18} 
			&&\left|\sum_{i\in F_n} \left(g_n \left(\frac{i}{n}\right) -g_n \left(\frac{i+1}{n}\right)\right)\int_{0}^{i/n}  P_n(b, \varepsilon,x)dx\right| \\
			&&= \frac{1}{n} \sum_{i\in F_n} \left| \int_{0}^{i/n} P_n(b, \varepsilon,x)dx\right|.  \notag 
		\end{eqnarray}
		Using equations~\eqref{eq17} and \eqref{eq18}, we obtain
		
		\begin{eqnarray}
			\label{eq19} \left|h_1\right| &=& \left| \sum_{i=1}^{n-1}\left(g_n\left(\frac{i}{n}\right)-g_n\left(\frac{i+1}{n}\right)\right)\int_{0}^{i/n}P_n(b, \varepsilon, x)dx  \right| \notag \\
			&=&  \left|\sum_{i\in F_n} \left(g_n \left(\frac{i}{n}\right) -g_n \left(\frac{i+1}{n}\right)\right)\int_{0}^{i/n}  P_n(b, \varepsilon,x)dx\right.  \\
			&+& \left.\sum_{i\in E_n} \left(g_n \left(\frac{i}{n}\right) -g_n \left(\frac{i+1}{n}\right)\right)\int_{0}^{i/n}  P_n(b, \varepsilon,x)dx\right| \notag \\
			&\geq& \left|\frac{1}{n} \sum_{i\in F_n} \left|\int_{0}^{i/n}  P_n(b, \varepsilon,x)dx\right| - \frac{1}{n} \sum_{i\in E_n} \left|\int_{0}^{i/n}  P_n(b, \varepsilon,x)dx\right| \right| \notag\\
			&=& \left|\frac{1}{n} \sum_{i=1}^{n-1} \left|\int_{0}^{i/n}  P_n(b, \varepsilon,x)dx\right| - \frac{2}{n} \sum_{i\in E_n} \left|\int_{0}^{i/n}  P_n(b, \varepsilon,x)dx\right| \right|  \notag \\
			&\geq& M_n(b,\varepsilon) - O\left(\frac{1}{\sqrt{n}}\right).\notag
		\end{eqnarray}
		
		Finally, from equality \eqref{eq14} and by account of \eqref{eq15}, \eqref{eq16} and \eqref{eq19} we obtain
		\begin{eqnarray*}
			\left|\int_{0}^{1} g_n(x) P_n(b, \varepsilon,x)dx\right| \geq M_n(b,\varepsilon) - O\left(\frac{1}{\sqrt{n}}\right)-O(1).
		\end{eqnarray*}
		Thus, by the condition of the Theorem \ref{theorem2} we get
		\begin{eqnarray} \label{eq20}
			&&\limsup_{n \rightarrow +\infty}\left|\int_{0}^{1} g_n(x) P_n(b, \varepsilon,x)dx\right| = \limsup_{n \rightarrow +\infty} M_n(b,\varepsilon) =+\infty. 
		\end{eqnarray}
		
		Now we must consider the sequence of bounded and continuous linear functionals on $Lip_1$ (see \eqref{eq7})
		\begin{eqnarray*}
			B_n(f):= \int_{0}^{1} f(x)  P_n(b, \varepsilon,x)dx.
		\end{eqnarray*}
		So, from \eqref{eq20} we obtain
		\begin{eqnarray*}
			\limsup_{n \rightarrow +\infty}\left|B_n \left(g_n\right)\right| =+ \infty.
		\end{eqnarray*}
		Since
		\begin{eqnarray*}
			\left|\left|g_n\right|\right|_{Lip_1} = 	\left|\left|g_n\right|\right|_C +\sup_{x,y \in \left[0,1\right]}  \frac{\left|g_n(x)-g_n(y)\right|}{\left|x-y\right|} \leq 2,
		\end{eqnarray*}
		by the Banach-Steinhaus theorem, there exists a function $g\in Lip_1$ such that
		\begin{eqnarray*}
			\limsup_{n \rightarrow +\infty}\left|B_n \left(g\right)\right| =+ \infty.
		\end{eqnarray*}
		
		Consequently
		\begin{eqnarray*}
			\limsup_{n \rightarrow +\infty}\left|\int_{0}^{1} g(x)  P_n(b, \varepsilon,x)dx\right| 
			&=&	\limsup_{n \rightarrow +\infty}\left|\sum_{k=1}^{n} b_k \int_{0}^{1} g(x)  \varphi_{k}(x)dx\right| \\ 
			&=& \limsup_{n \rightarrow +\infty}\left|\sum_{k=1}^{n} b_k C_k(g)\right| =+ \infty.
		\end{eqnarray*}
		Hence the series 
		$$\sum_{k=1}^{\infty} b_k C_k(g)$$
		is divergent for $(b_k)\in l_{q(\varepsilon)} .$ So, it is evident that  $(C_k (g))\notin l_{p(\varepsilon)}$ for any $\varepsilon \in (0,1)$ or
		\begin{eqnarray*}
			\sum_{k=1}^{\infty} \left|C_n(g)\right|^{2-\varepsilon} =+\infty.
		\end{eqnarray*}
		
	\end{proof}
	\begin{corollary}
		\label{cor1} Let $(\varphi_n)$ be an ONS. Suppose that for some $\varepsilon \in (0,1)$  $(C_n (l)) \in l_{p(\varepsilon)},$ where $l(x)=1,$ $x \in [0,1].$ If for arbitrary $(a_n) \in l_{q(\varepsilon)}$ 
		\begin{equation*}
			M_n(a, \varepsilon) = O_{\varepsilon}(1),
		\end{equation*}
		then for any $f\in Lip_1$
		
		\begin{eqnarray*}
			\sum_{n=1}^{\infty} C_n(f) \varphi_{n}(x)
		\end{eqnarray*}
		is unconditional convergent a. e. on $[0,1].$
	\end{corollary}
	
	The validity of Corollary \ref{cor1} derives from Theorems \ref{theorem1} and Theorem {B}.
	\begin{theorem}\label{theorem4}
		Let $(d_n)$ be a nondecreasing sequence of positive numbers. Any ONS $(\varphi_{n})$ contains a subsystem $(\varphi_{n_{k}})$ such that 
		\begin{eqnarray}
			\label{eq22} M_n(d a,\varepsilon)=O_{\varepsilon}(1)
		\end{eqnarray}
		for any $\varepsilon \in (0,1)$  and $(a_n) \in l_{q(\varepsilon)},$ where $da=(d_n  a_n).$
	\end{theorem}
	
	\begin{proof}
		Let ONS $(\varphi_{n})$ be a complete system on $[0,1].$ Then, by the Parceval's equality for any $x \in [0,1]$
		\begin{eqnarray*}
			\sum_{n=1}^{\infty} \left(\int_{0}^{x} \varphi_{n} (u) du \right)^2 = x.
		\end{eqnarray*}
		According to Dini's theorem, there exists a subsequence of natural numbers \((n_k)\) such that the convergence is uniform with respect to \(x \in [0,1]\).
		
		\begin{eqnarray*}
			\sum_{s=n_k}^{\infty} \left(\int_{0}^{x} \varphi_{s} (u) du \right)^2 \leq \frac{1}{k^4}.
		\end{eqnarray*}
		From here 
		\begin{eqnarray} \label{eq23}
			\left|\int_{0}^{x} \varphi_{n_k} (u) du \right| \leq \frac{1}{d_k k^2}, \ \ \ k=1, 2, \dots
		\end{eqnarray}
		uniformly with respect to $x\in [0,1].$ Next, in our case 
		\begin{eqnarray*}
			P_n(da, \varepsilon, x) := \sum_{k=1}^{n} d_k a_k \varphi_{n_k} (x),
		\end{eqnarray*}
		then, by using \eqref{eq23}
		\begin{eqnarray*}
			\left|M_n(d a,\varepsilon)\right|&=& \frac{1}{n} \sum_{i=1}^{n-1} \left| \int_{0}^{i/n} P_n(da, \varepsilon, x) dx \right| \\
			&=& \frac{1}{n} \sum_{i=1}^{n-1} \left| \sum_{k=1}^{n} d_k a_k \int_{0}^{i/n}\varphi_{n_k} (x) dx  \right| \leq \frac{1}{n} \sum_{i=1}^{n-1}  \sum_{k=1}^{n} d_k \left|a_k\right|  \frac{1}{d_k k^2}  \\
			&\leq& \left(\sum_{k=1}^{n} \left|a_k\right| ^{p(\varepsilon)}\right)^{1/{p(\varepsilon)}} \left(\sum_{k=1}^{n} k^{-2q(\varepsilon)}\right)^{1/q(\varepsilon)} =O_{\varepsilon}(1).
		\end{eqnarray*}
		
	\end{proof}

	\begin{theorem}
		\label{theorem5} Let $(d_n)$ be a nondeacrising sequance of positive numbers and $d_n=O(\sqrt{n}).$ Any ONS $(\varphi_n)$ contains a subsystem $(\varphi_{n_k})$ such that the series 
		\begin{eqnarray*}
			\sum_{k=1}^{\infty}d_k C_{n_k} (f) \varphi_{n_{k}} (x)
		\end{eqnarray*}
		is unconditionally convergent a.e. on $[0,1]$ for an arbitrary $f\in Lip_1.$
	\end{theorem}
	
	\begin{proof}
		Here we also consider the very same subsystem, which was investigated in Theorem \ref{theorem4}. That is why for any $\varepsilon \in (0,1)$  and  $(a_n)\in l_{q(\varepsilon)}$ we get $(da=(d_n a_n))$ condition \eqref{eq22}.

		Also as $p(\varepsilon)=2-\varepsilon,$  if $l(x)=1,$ $x\in [0,1]$  (see \eqref{eq23})
		\begin{eqnarray}
			\label{eq25}  \sum_{k=1}^{\infty} \left|d_k C_{n_k} (l) \right|^{p(\varepsilon)} &=& \sum_{k=1}^{\infty} \left|d_k \int_{0}^{1} \varphi_{n_{k}}(x) dx \right|^{p(\varepsilon)} \\
			&<& \sum_{k=1}^{\infty} \left|d_k \frac{1}{d_k k^2}\right|^{p(\varepsilon)} < +\infty. \notag\ \ 
		\end{eqnarray} 
		Now we rewrite equality \eqref{eq6} for $P_n (da,\varepsilon,x)$  and we obtain
		
		\begin{eqnarray}
			\label{eq26} \int_{0}^{1} f(x) P_n (d a, \varepsilon, x) dx 
			&&= \sum_{i=1}^{n-1}\left(f\left(\frac{i}{n}\right)-f\left(\frac{i+1}{n}\right)\right)\int_{0}^{i/n}P_n(da, \varepsilon, x)dx \notag  \\
			&&+ \sum_{i=1}^{n}\int_{\left(i-1\right)/n}^{i/n}\left(f\left(x\right) -f\left(\frac{i}{n}\right)\right)P_n(da, \varepsilon, x) dx    \\
			&&+ f\left(1\right)\int_{0}^{1}P_n(da, \varepsilon, x)dx=S_1 +S_2 +S_3.  \notag  
		\end{eqnarray}
		According to \eqref{eq22}, as $f\in Lip1$ 
		\begin{eqnarray} \label{eq27}
			\left|S_1\right| &=& \left|\sum_{i=1}^{n-1}\left(f\left(\frac{i}{n}\right)-f\left(\frac{i+1}{n}\right)\right)\int_{0}^{i/n}P_n(da, \varepsilon, x)dx \right|  \\
			&=& O(1) \frac{1}{n} \sum_{i=1}^{n-1} \left|\int_{0}^{i/n}P_n(da, \varepsilon, x)dx \right| \notag  \\
			&=& O(1) M_n (da, \varepsilon) = O_{\varepsilon}(1). \notag
		\end{eqnarray} 
		By \eqref{eq25}, as $(a_k) \in l_{q(\varepsilon)}$ we have
		\begin{eqnarray} \label{eq28}
			\left|S_3\right| &=& \left|f\left(1\right)\int_{0}^{1}P_n(da, \varepsilon, x)dx\right| = O(1)  \left|\int_{0}^{1} \sum_{i=1}^{n}d_k a_k \varphi_{n_{k}}(x) dx \right| \notag\\
			&=& O(1)  \left| \sum_{i=1}^{n}d_k a_k \int_{0}^{1}\varphi_{n_{k}}(x) dx \right| 
			= O(1)  \sum_{i=1}^{n}d_k  \left|a_k \right| \frac{1}{d_k k^2} = O_{\varepsilon}(1). 
		\end{eqnarray}
		Finally, by Lemma 2, since $f\in Lip_1,$ $d_n=O(\sqrt{n})$ and $(a_n) \in l_{q(\varepsilon)}$  we obtain
		
		\begin{eqnarray}
			\label{eq29} \left|S_2\right| &=& \left|\sum_{i=1}^{n}\int_{\left(i-1\right)/n}^{i/n}\left(f\left(x\right) -f\left(\frac{i}{n}\right)\right)P_n(da, \varepsilon, x) dx \right| \notag\\
			&=& O(1) \frac{1}{n} \sum_{i=1}^{n} \int_{\left(i-1\right)/n}^{i/n}\left|P_n(da, \varepsilon, x) \right| dx \\
			&=& O(1) \frac{1}{n}  \left(\int_{0}^{1}P_n^2(da, \varepsilon, x) dx \right)^{1/2} \notag = O\left(\frac{1}{n}\right) \left(\sum_{k=1}^{n} d_k^2 a_k^2 \right)^{1/2} \notag \\
			&=& O(1)\, d_n\, n^{-1/2 - 1/q(\varepsilon)} \left( \sum_{k=1}^n a_k^{q(\varepsilon)} \right)^{1/q(\varepsilon)}
			= O_{\varepsilon}(1).  \notag 
		\end{eqnarray}
		
		To reach the final conclusion we combine the findings in \eqref{eq27}, \eqref{eq28} and \eqref{eq29} inequality \eqref{eq26} which yields the result
		\begin{eqnarray*}
			\left|\int_{0}^{1} f(x)P_n(da, \varepsilon, x) dx \right|= O_{\varepsilon}(1)
		\end{eqnarray*}
		and consequently
		\begin{eqnarray}
			\label{eq30} \left|\sum_{k=1}^{n} d_k a_k C_k(f) \right| &=& \left|\sum_{k=1}^{n} d_k a_k \int_{0}^{1} f(x) \varphi_{n_{k}}(x) dx \right| \\ 
			&=& 	\left|\int_{0}^{1} f(x)P_n(da, \varepsilon, x) dx \right|= O_{\varepsilon}(1). \notag
		\end{eqnarray}
		
		From here we conclude that the series 
		\begin{eqnarray*}
			\sum_{k=1}^{\infty} d_k a_k C_k(f)
		\end{eqnarray*}
		is convergent for any $(a_n)\in l_{q(\varepsilon)}.$ Thus   $(d_k C_k (f))\in l_{p(\varepsilon)}$  or
		\begin{eqnarray*}
			\sum_{k=1}^{\infty} \left|d_k C_{n_k}(f)\right|^{2- \varepsilon} < + \infty.
		\end{eqnarray*}
		As we know, by Theorem B the series 
		\begin{eqnarray*}
			\sum_{k=1}^{\infty} d_k C_{n_k}(f)\varphi_{n_{k}}(x)
		\end{eqnarray*}
		is unconditionally convergent a. e. for any $f\in Lip_1.$
		
	\end{proof}
	
	Here a question might arise: is the condition (see Theorem \ref{theorem1})
	\begin{eqnarray*}
		C_n(l) \in l_{p(\varepsilon)}, \text{ where } l(x) = 1, x \in [0,1]
	\end{eqnarray*}
	sufficient for convergence of series
	\begin{eqnarray*}
		\sum_{k=1}^{\infty} \left| C_{k}(f)\right|^{2- \varepsilon}
	\end{eqnarray*}
	for any $f\in Lip_1 \ ?$
	
	\begin{theorem}
		\label{theorem6} 
		There exists a function $g \in Lip_1$ and ONS $(\Phi_n)$  such that 
		\begin{flalign*}
			1) \int_{0}^{1} \Phi_n (x) dx =0,  \ \ \ n=1, 2, \dots . && 
		\end{flalign*}
		\begin{flalign*}
			2) \ \text{For any } \varepsilon > 0, && 
		\end{flalign*}
		\begin{eqnarray*}
			\sum_{n=1}^{\infty} \left|C_n(g, \Phi)\right| ^{p(\varepsilon)} =+ \infty, 
		\end{eqnarray*}
		where 
		\begin{eqnarray*}
			C_n(g,\Phi)= \int_{0}^{1} g(x) \Phi_n (x) dx =0,  \ \ \ (n=1, 2, \dots).
		\end{eqnarray*}
		
	\end{theorem}
	
	\begin{proof}
		Let us assume that $f(x)=1-\cos 4(x- 1/2) \pi.$ According to the Banach Theorem there exist an ONS $(\varphi_n),$ such that 
		\begin{eqnarray}
			\label{eq31} \limsup_{n \rightarrow +\infty} \left|S_n (x, f)\right| =+\infty.
		\end{eqnarray}
		From here for any $\varepsilon >0$
		\begin{eqnarray}\label{eq32}
			\sum_{n=1}^{\infty} \left|C_n(f)\right| ^{p(\varepsilon)} =+ \infty, 
		\end{eqnarray}
		where 
		\begin{eqnarray*}
			C_n(f)= \int_{0}^{1} f(x) \varphi_n (x) dx =0,  \ \ \ (n=1, 2, \dots).
		\end{eqnarray*}
		
		Now we investigate the function
		\begin{equation} \label{eq33}
			g\left( x\right) =\left\{ 
			\begin{array}{ccc}
				f(2x), & \text{when} & x\in \left[ 0,\frac{1}{2}\right], \\ 
				\\
				0, & \text{when} & x\in \left[ \frac{1}{2},1\right].
			\end{array}%
			\right.
		\end{equation}
		It is evident that $g\in Lip_1.$
		
		Suppose  
		\begin{equation} \label{eq34}
			\Phi_n\left( x\right) =\left\{ 
			\begin{array}{ccc}
				\varphi_{n}(2x), & \text{when} & x\in \left[ 0,\frac{1}{2}\right], \\ 
				\\
				- \varphi_{n} \left(2\left(x-\frac{1}{2}\right)\right), & \text{when} & x\in \left[ \frac{1}{2},1\right].
			\end{array}%
			\right.
		\end{equation}
		Then
		\begin{eqnarray*}
			\int_{0}^{1} \Phi_n (x) dx &=& \int_{0}^{1/2} \varphi_{n}(2x) dx -  \int_{1/2}^{1} \varphi_{n} \left(2\left(x-\frac{1}{2}\right)\right) dx
		\end{eqnarray*}
		
		\begin{eqnarray*}
			= \frac{1}{2} \int_{0}^{1} \varphi_{n}(x) dx - \frac{1}{2} \int_{0}^{1} \varphi_{n}(x) dx = 0.
		\end{eqnarray*}
		Next (see \eqref{eq34})
		\begin{eqnarray*}
			\int_{0}^{1} \Phi_n (x) \Phi_m (x)dx &=& \int_{0}^{1/2} \varphi_{n}(2x) \varphi_{m}(2x)dx \\ &+&  \int_{1/2}^{1} \varphi_{n} \left(2\left(x-\frac{1}{2}\right)\right) \varphi_{m} \left(2\left(x-\frac{1}{2}\right)\right)dx\\ &=& \frac{1}{2} \int_{0}^{1} \varphi_{n}(x) \varphi_{m}(x)  dx + \frac{1}{2} \int_{0}^{1} \varphi_{n}(x) \varphi_{m}(x)  dx = 0.
		\end{eqnarray*}
		
		Also (see \eqref{eq34})
		\begin{eqnarray*}
			\int_{0}^{1} \Phi_n^2 (x) dx &=& \int_{0}^{1/2} \varphi_{n}^2(2x) dx +  \int_{1/2}^{1} \varphi_{n}^2 \left(2\left(x-\frac{1}{2}\right)\right) dx \\
			&=& \frac{1}{2} \int_{0}^{1} \varphi_{n}^2(x) dx + \frac{1}{2} \int_{0}^{1} \varphi_{n}^2(x) dx =1.
		\end{eqnarray*}
		Thus, $(\Phi_n)$ is an ONS.
		
		Finally (see \eqref{eq33}, \eqref{eq34})
		\begin{eqnarray*}
			C_n(g,\Phi) &=&  \int_{0}^{1} g(x) \Phi_n(x) dx= \int_{0}^{1/2}f(2x) \varphi_{n}(2x) dx \\
			&=&  \frac{1}{2} \int_{0}^{1}  f(x)\varphi_{n}(x) dx =\frac{1}{2} C_n(f).
		\end{eqnarray*}
		From where, according to \eqref{eq32} we conclude that
		\begin{eqnarray*}
			\sum_{n=1}^{\infty} \left|C_n(g,\Phi)\right|^{p(\varepsilon)} = 2^{- p (\varepsilon)} \sum_{n=1}^{\infty} \left|C_n(f)\right|^{p(\varepsilon)} = + \infty.
		\end{eqnarray*}
		Now we can prove that a.e. on $[0,1]$
		\begin{eqnarray*}
			\limsup_{n \rightarrow +\infty} \left|Q_n(x,g)\right| = \limsup_{n \rightarrow +\infty} \left|\sum_{k=1}^{n} C_k(g,\Phi)\Phi_k (x)\right|=+\infty.
		\end{eqnarray*}
		Indeed, if 
		\begin{eqnarray*}
			S_n(x,f) := \sum_{k=1}^{n} C_k(f) \varphi_{k}(x) 
		\end{eqnarray*}
		then
		\begin{equation*}
			Q_n\left( x,g\right) =\left\{ 
			\begin{array}{ccc}
				\frac{1}{2}S_n(2x,f), & \text{when} & x\in \left[ 0,\frac{1}{2}\right], \\ 
				\\
				\frac{1}{2}S_n(2(x-\frac{1}{2}),g), & \text{when} & x\in \left[ \frac{1}{2},1\right].
			\end{array}%
			\right.
		\end{equation*}
		
		Let us assume that $x\in \left[0,  \frac{1}{2}\right],$ then
		\begin{eqnarray*}
			Q_n\left( x,g\right) = \frac{1}{2}S_n(2x,f) = \frac{1}{2} \sum_{k=1}^{n} C_k(f) \varphi_{k}(2x) 
		\end{eqnarray*}
		or if $2x=t,$ $t\in [0,1]$
		\begin{eqnarray*}
			Q_n \left(\frac{t}{2}, g\right) =  \frac{1}{2} S_n(t,f)=\frac{1}{2} \sum_{k=1}^{n} C_k(f) \varphi_{k}(t). 
		\end{eqnarray*}
		By \eqref{eq31} we get
		\begin{eqnarray*}
			\limsup_{n \rightarrow +\infty} \left|	Q_n \left(\frac{t}{2}, g\right)\right| = \frac{1}{2} \limsup_{n \rightarrow +\infty} \left| S_n(t,f) \right| = + \infty,
		\end{eqnarray*}
		a.e. $t \in \left[0, 1 \right]. $
		
		Analogously is possible to show that, if $x\in \left[\frac{1}{2}, 1 \right],$ then 
		$$ Q_n \left(x, g\right)=\frac{1}{2} S_n \left(2 \left(x-\frac{1}{2}\right),f\right)=\frac{1}{2} \sum_{k=1}^{n} C_k(f) \varphi_{k}\left(2 \left(x-\frac{1}{2}\right)\right) $$
		or if $2 \left(x-\frac{1}{2}\right)=t,$ $t\in [0,1]$
		
		\begin{eqnarray*}
			Q_n \left(\frac{t}{2} + \frac{1}{2}, g\right) =  \frac{1}{2} S_n(t,f)=\frac{1}{2} \sum_{k=1}^{n} C_k(f) \varphi_{k}(t). 
		\end{eqnarray*}
		By \eqref{eq31} we get
		\begin{eqnarray*}
			\limsup_{n \rightarrow +\infty} \left|	Q_n \left(\frac{t}{2} + \frac{1}{2}, g\right)\right| = \frac{1}{2} \limsup_{n \rightarrow +\infty} \left| S_n(t,f) \right|= + \infty,
		\end{eqnarray*}
		a.e. $t \in \left[0, 1 \right]. $
		
		Consequently
		\begin{eqnarray*}
			\limsup_{n \rightarrow +\infty} \left|Q_n(t,g)\right| = + \infty, 
		\end{eqnarray*}
		a.e. on $[0,1],$ for $g \in Lip_1.$
	\end{proof}
	\section{Problems of efficiency }
	We call the condition $M_n(a,\varepsilon) = O(1)$ efficient if it is easily verified for classical ONS.
	
	\begin{theorem}
		\label{theorem7} Let $(\varphi_{n})$  be an ONS and  uniformly with respect to $x \in [0,1]$
		\begin{eqnarray*}
			\int_{0}^{x} \varphi_{n}(u) du = O(1) \frac{1}{n}. 
		\end{eqnarray*}
		Then for arbitrary $(a_n) \in l_{2-\varepsilon}$ $((a_n) \in l_2)$ for any $\varepsilon \in (0,1)$
		\begin{eqnarray}
			\label{eq35}
			M_n (a,\varepsilon)=O(1).
		\end{eqnarray}
		
	\end{theorem}
	
	\begin{proof}
		
		\begin{eqnarray}\label{eq36}
			M_n (a,\varepsilon)&=&\frac{1}{n} \sum_{i=1}^{n-1} \left| \int_{0}^{i/n} P_n(a,x) dx \right|  = \frac{1}{n} \sum_{i=1}^{n-1} \left|\sum_{k=1}^{n} a_k \int_{0}^{i/n} \varphi_{k}(x)dx \right|   \\
			&=& O(1) \sum_{k=1}^{n} \frac{1}{k} \left|a_k\right| = O(1)  \left(\sum_{k=1}^{n} a_k^2\right)^{1/2} \left(\sum_{k=1}^{n} k^{-2}\right)^{1/2} = O(1). \notag
		\end{eqnarray}
		So the trigonometric $(\sqrt{2} \cos 2 \pi nx, \sqrt{2} \sin 2 \pi nx)$ and the Walsh systems satisfy the condition \eqref{eq35}.
	\end{proof}
	
	\begin{theorem}
		\label{theorem8}
		If $(X_n)$ is the Haar system, then for any arbitrary $(a_n)\in l_{2-\varepsilon},$ the condition \eqref{eq35} holds for any $\varepsilon \in (0,1).$
	\end{theorem}
	
	\begin{proof}
		According to the definition of the Haar system we have (see \cite{Alexits} ch.I, \#6. p.54)
		\begin{eqnarray*} 
			\left|\int_{0}^{x} \sum_{k=2^m}^{2^{m+1}} a_k X_{k}(u) du\right| \leq 2^{-m/2} \left|a_{k(m)}\right|,
		\end{eqnarray*}
		where $2^m \leq k(m) < 2^{m+1}$.
		
		Using the  inequality \eqref{eq36}, when $n=2^p$ for an arbitrary $f\in Lip_1$ we get
		\begin{eqnarray*}
			V_n (a)&=&\frac{1}{n} \sum_{i=1}^{n-1} \left| \int_{0}^{i/n} P_n(a,x) dx \right| = \frac{1}{2^p} \sum_{i=1}^{2^p-1} \left|\sum_{m=0}^{p-1} \int_{0}^{i/{2^p}} \sum_{k=2^m}^{2^{m+1}-1}X_k(x)a_k dx\right|  \\
			&=& O(1) \sum_{m=0}^{p-1} 2^{-m/2} \left|a_{k(m)}\right| = O(1) \sum_{m=0}^{p-1} 2^{-m/2} \left(  \sum_{k=2^m}^{2^{m+1}-1}  a_{k}^{2} \right)^{1/2} \\
			&=&O(1) \left(\sum_{m=0}^{p-1}  \sum_{k=2^m}^{2^{m+1}-1}  a_{k}^{2} \right)^{1/2} \left(\sum_{m=0}^{p} 2^{-m}\right)^{1/2}=  O(1).
		\end{eqnarray*}
		
		Finally, we conclude that when $n=2^p+l,$ $1\leq l<2^p$ the condition \eqref{eq35} is valid.
		
	\end{proof}

	\section{Conclusion}
	From the discussions presented throughout this article, it is evident that, although the general Fourier series of functions in the \(Lip_1\) class do not converge in the usual sense, it is nevertheless possible to identify a subclass of orthonormal systems whose elements satisfy specific conditions, with respect to which the Fourier series of \(Lip_1\) functions converge unconditionally (see Theorem~\ref{theorem1.1}). Furthermore, we have established that the conditions imposed on the functions of the orthonormal system are both precise and valid.
	
	Moreover, it is worth noting that every orthonormal system contains a subsystem with respect to which the general Fourier series of any function \(f \in Lip_1\) converge unconditionally almost everywhere on the interval \([0,1]\).\\

	\textbf{Acknowledgement:}
	The authors wish to extend their deepest gratitude to their esteemed colleagues and collaborators for their invaluable contributions throughout the course of this research. Their insightful discussions, constructive criticism, and unwavering support have played a pivotal role in shaping the development and successful completion of this work.
	
	We would also like to sincerely acknowledge the anonymous reviewers for their meticulous and thoughtful examination of the manuscript. Their careful scrutiny and insightful comments have significantly enhanced the clarity, coherence, and overall quality of the paper.
	
	Moreover, we express our profound appreciation to the editorial board and reviewers of Publicationes Mathematicae Debrecen for their dedication to upholding rigorous academic standards. Their ongoing commitment to excellence fosters a scholarly environment that greatly facilitates the advancement of research within the mathematical sciences.

\end{document}